\documentclass[11pt,reqno,tbtags]{amsart}
\usepackage{amssymb}
\usepackage{ifdraft}
\usepackage{url}
\usepackage[square,numbers]{natbib}

\numberwithin{equation}{section}

\allowdisplaybreaks

\newtheorem{theorem}{Theorem}[section]
\newtheorem{lemma}[theorem]{Lemma}

\newtheorem*{claim}{Claim}

\theoremstyle{definition}

\newtheorem{definition}[theorem]{Definition}
\newtheorem{definitions}[theorem]{Definitions}
\newtheorem{problem}[theorem]{Problem}
\newtheorem{remark}[theorem]{Remark}

\theoremstyle{remark}

\newenvironment{romenumerate}[1][0pt]{
\addtolength{\leftmargini}{#1}\begin{enumerate}
 }{\end{enumerate}}

\newcounter{oldenumi}
{\setcounter{oldenumi}{\value{enumi}}
\begin{romenumerate} \setcounter{enumi}{\value{oldenumi}}}
{\end{romenumerate}}

\newcounter{thmenumerate}
\newenvironment{thmenumerate}
{\setcounter{thmenumerate}{0}%
 \def\item{\par
 \refstepcounter{thmenumerate}\textup{(\roman{thmenumerate})\enspace}}
}
{}

\newcounter{romxenumerate}   

\newcounter{xenumerate}   

\newcommand{\refT}[1]{Theorem~\ref{#1}}

\newcommand{\refL}[1]{Lemma~\ref{#1}}
\newcommand{\refR}[1]{Remark~\ref{#1}}
\newcommand{\refS}[1]{Section~\ref{#1}}

\newcommand\REM[1]{{\raggedright\texttt{[#1]}\par\marginal{XXX}}}

\newenvironment{comment}{\setbox0=\vbox\bgroup}{\egroup} 

\begingroup
  \divide\count255 by 60
  \multiply\count255 by -60
  \advance\count255 by \time
  \ifnum \count255 < 10 \xdef\klockan{\the\count1.0\the\count255}
  \else\xdef\klockan{\the\count1.\the\count255}\fi
\endgroup

\newcommand\set[1]{\ensuremath{\{#1\}}}

\newcommand\bigpar[1]{\bigl(#1\bigr)}
\newcommand\Bigpar[1]{\Bigl(#1\Bigr)}

\def\rompar(#1){\textup(#1\textup)}    

\newcommand\Bigparfrac[2]{\Bigpar{\frac{#1}{#2}}}

\def\xexp(#1){e^{#1}}
\newcommand\ceil[1]{\lceil#1\rceil}
\newcommand\floor[1]{\lfloor#1\rfloor}

\newcommand\ntoo{\ensuremath{{n\to\infty}}}

\newcommand\punkt{.\spacefactor=1000}    

\newcommand\ie{i.e\punkt}
\newcommand\eg{e.g\punkt}

\newcommand\cf{cf\punkt}

\newcommand\whp{w.h.p\punkt}

\newcommand{\tend}{\longrightarrow}

\newcommand\pto{\overset{\mathrm{p}}{\tend}}

\newcounter{CC}
\newcommand{\CC}{\stepcounter{CC}\CCx} 
\newcommand{\CCx}{C_{\arabic{CC}}}     
\newcommand{\CCdef}[1]{\xdef#1{\CCx}}     
\newcounter{cc}

\newcommand\E{\operatorname{\mathbb E{}}}
\renewcommand\P{\operatorname{\mathbb P{}}}
\newcommand\Var{\operatorname{Var}}
\newcommand\Cov{\operatorname{Cov}}

\newcommand\gD{\Delta}

\newcommand\gam{\gamma}

\newcommand\gl{\lambda}

\newcommand\eps{\varepsilon}

\renewcommand\phi{\xxx}  

\newcommand\cA{\mathcal A}

\newcommand\cM{\mathcal M}

\newcommand\qw{^{-1}}
\newcommand\qww{^{-2}}

\renewcommand{\=}{=}

\renewcommand\ij{_{ij}}
\newcommand\ijto[1]{_{i,j=1}^{#1}}
\newcommand\tp{T^{\mathrm p}}
\newcommand\ts{T^{\mathrm s}}
\newcommand\tps{T^{\mathrm {ps}}}
\newcommand\tx{T^{*}}
\newcommand\xxcm[2]{$(#1,#2)$-corner matrix}
\newcommand\xxcmm[2]{$(#1,#2)$-corner matrices}
\newcommand\xxcsm[2]{$(#1,#2)$-corner submatrix}
\newcommand\xxc[2]{$(#1,#2)$-corner}
\newcommand\klcm{\xxcm{k}{\ell}}
\newcommand\klcmm{\xxcmm{k}{\ell}}
\newcommand\klcsm{\xxcsm{k}{\ell}}
\newcommand\mlc{\xxc{m}{\ell}}
\newcommand\logq{\log_Q}
\newcommand\yml{Y_{m,\ell}}
\newcommand\qa[1]{\ensuremath{#1\times#1}}
\newcommand\iii[2]{\ensuremath{[#1,#2]}}
\newcommand\imx{I_{M'}}

\newcommand\urladdrx[1]{{\urladdr{\def~{{\tiny$\sim$}}#1}}}

\begin{document}
\title[Superboolean rank and the largest triangular submatrix]
{Superboolean rank and the size of the largest triangular submatrix of a
random matrix}

\date{2 September, 2011}

\author{Zur Izhakian}
\address{School of Mathematical Sciences, Tel Aviv
   University, Ramat Aviv,  Tel Aviv 69978, Israel.
\vskip0.001pt Department of Mathematics, Bar-Ilan University,
Ramat-Gan 52900, Israel.} \email{zzur@math.biu.ac.il}

\thanks{The research of the first author has  been  supported  by the
Israel Science Foundation (ISF grant No.  448/09) and  by the
Oberwolfach Leibniz Fellows Programme (OWLF), Mathematisches
Forschungsinstitut Oberwolfach, Germany.}

\author{Svante Janson}
\address{Department of Mathematics, Uppsala University, PO Box 480,
SE-751~06 Uppsala, Sweden}
\email{svante.janson@math.uu.se}
\urladdrx{http://www2.math.uu.se/~svante/}

\author{John Rhodes}
\address{Department of Mathematics, University of California, Berkeley,
970 Evans Hall \#3840, Berkeley, CA 94720-3840 USA.}
\email{blvdbastille@aol.com;rhodes@math.berkeley.edu}

\thanks{\textbf{Acknowledgement:}
  This work was started during a chance meeting of researchers from two
different groups at a supper table in
Mathematisches Forschungsinstitut Oberwolfach (MFO), Germany in April 2011,
and the work was essentially completed during the authors' stay at MFO.
We thank other MFO visitors, in particular Gabor Lugosi, for helpful comments.
}


\begin{comment}  
05 Combinatorics
05C Graph theory [For applications of graphs, see 68R10, 90C35, 94C15]
05C05 Trees
05C07 Vertex degrees
05C35 Extremal problems [See also 90C35]
05C40 Connectivity
05C65 Hypergraphs
05C80 Random graphs
05C90 Applications
05C99 None of the above, but in this section

60 Probability theory and stochastic processes
60C Combinatorial probability
60C05 Combinatorial probability

60F Limit theorems [See also 28Dxx, 60B12]
60F05 Central limit and other weak theorems
60F17 Functional limit theorems; invariance principles

60G Stochastic processes
60G09 Exchangeability
60G55 Point processes

\end{comment}

\begin{abstract}
We explore the size of the largest (permuted)
triangular submatrix of a random matrix, and more precisely its
asymptotical behavior as the size of
the ambient matrix tends to infinity. The importance of such
permuted triangular submatrices arises when dealing with certain
combinatorial algebraic settings in which these submatrices
determine the rank of the ambient matrix, and thus attract a
special attention.
\end{abstract}

\maketitle


\section{Introduction}\label{S:intro}

Let $X=(x\ij)_{i,j=1}^n$ be a random $n\times n$ matrix.
We assume that the entries of $X$ are taken from some set $\cA$
and that they are independent and identically distributed, with
$\P(x\ij=a)=p_a$
for some fixed probabilities $p_a$, $a\in\cA$.
We assume further that $0,1\in\cA$ and $p_0,p_1>0$.

The purpose of the present paper is to study the size of the
largest triangular submatrix of $X$, and more precisely
its asymptotical behavior as \ntoo. We actually consider four
versions of this problem; it turns out that to the first order
studied here, they all have the same answer.

\begin{definitions}
  \begin{thmenumerate}
  \item
A \emph{submatrix} of a matrix $A=(a\ij)_{i\in M,j\in N}$ is any
matrix obtained by deleting rows  and/or columns of $A$. In other
words, it is a matrix $(a\ij)_{i\in I,j\in J}$ for a non-empty set
of rows $I \subseteq M$ and a non-empty set of columns $J
\subseteq N$. (We preserve the order of the rows and columns in
$I$ and $J$.)

\item
A \emph{permutation} of a matrix is a matrix obtained by a
permutation of the rows and a (possibly different) permutation of the columns.
In particular, a \emph{permuted submatrix} of $(a\ij)$ is
$(a_{i_rj_s})_{r,s=1}^{k,\ell}$
for a sequence of distinct rows $i_1,\dots,i_k$ and
a sequence of distinct columns $j_1,\dots,j_\ell$.
\item
A \emph{(lower) triangular matrix} is a square matrix $(a\ij)\ijto
m$ such that $a\ij=0$ when ~$i<j$.
\item
A  \emph{special triangular matrix} is a square matrix
$(a\ij)\ijto m$ such that $a\ij=0$ when $i<j$ and $a_{ij}=1$ when
~$i=j$. (The remaining entries are arbitrary.)
  \end{thmenumerate}
\end{definitions}

Note that a $k\times\ell$ submatrix is determined by two \emph{sets} $I,J$
of indices with $|I|=k$, $|J|=\ell$, while a permuted submatrix is
determined by two
\emph{sequences}  $i_1,\dots,i_k$ and $j_1,\dots,j_\ell$ of indices, with each
sequence without repetitions.

We define the random variable $T_n$ as the maximal size (= number of rows,
or columns)
of a submatrix of
$X$ that is triangular; similarly
$\ts_n$, $\tp_n$, $\tps_n$ are the maximal sizes of a submatrices that are
special triangular, permuted triangular and permuted special triangular,
respectively. Equivalently, $\tp_n$ $[\tps_n]$
is the maximal size of a permuted submatrix of $X$
that is [special] triangular.
Note that
\begin{align}\label{tineq}
  \ts_n\le T_n\le\tp_n
\qquad\text{and}\qquad
  \ts_n\le \tps_n\le\tp_n.
\end{align}


The general motivation  for studying these quantities comes from
boolean algebra or, more generally, from (tropical) max-plus
algebra \cite{ABG} and supertropical algebra \cite{Iz}.
 These algebras take place over semirings and are fundamentally connected to graph theory, in particular matrices over these
 semirings   correspond uniquely to weighted directed graphs.
With  this correspondence, basic algebraic notions are naturally
substituted by combinatorial ones;
 for example, the role of the determinant is replaced  by the
 permanent. These combinatorial analogous also help to bypass the lack of negation in the
 ground semirings.
 As a consequence, computational complexity,  such as computing the rank of a
 matrix, is not always polynomial and
 could be NP-complete
\cite{Kim}  over this framework.

The specific motivation occurs if one considers either the boolean
case ($\cA=\set{0,1}$) or the superboolean case
($\cA=\set{0,1,1^\nu}$), the simplest example for a supertropical
semiring \cite{IRmatroid,IzhakianRowen2007SuperTropical}. These
papers lead to a new algebraic theory of combinatorics by
representing matroids by boolean matrices. In this theory, a
square matrix is non-singular if and only if it is permuted
special triangular, and the rank of a matrix is thus the maximal
size of a permuted special triangular submatrix, see
\citet{IRmatroid,IRmatroidII,IRLattice} for details. Consequently,
the rank of the random matrix $X$ is $\tps_n$.

\begin{theorem}
  \label{T1}
Let $Q\=1/p_0>1$, and let $\tx_n$ be any of
$T_n,\ts_n,\tp_n,\tps_n$. Then, as~\ntoo,
\begin{equation}\label{t1}
  \tx_n/\logq n\pto 2+\sqrt2,
\end{equation}
where $\pto$ denotes convergence in probability.
\end{theorem}

We say that an event occurs \emph{with high probability (\whp)} if its
probability tends to 1 as \ntoo.
Recall that, by the definition of convergence in probability, \eqref{t1}
says that for any $\eps>0$, \whp
\begin{equation}\label{t1x}
  (2+\sqrt2-\eps)\logq n < \tx_n <(2+\sqrt2+\eps)\logq n.
\end{equation}
By \eqref{tineq}, it suffices to prove the upper inequality for $\tp_n$ and
the lower for $\ts_n$.
The upper inequality is proved in \refS{Supper} and the lower in \refS{Slower};
the proofs are based on the first and second moment methods. (See \eg{}
\cite[p.~54]{JLR} for a general description of these methods.)

\begin{remark}
The corresponding problem of the largest square submatrix with
only 0's (or, equivalently, after interchange of 0 and 1, with
only 1's) has been studied by several authors, see \cite{SN} and
the references therein. It is shown in \cite{SN} that if $S_n$ is
the size of the largest such matrix, then $S_n/\logq n\pto 2$.
This problem can be seen as finding the largest balanced complete
subgraph of a random bipartite graph; the analogous problem of
finding the largest complete set in a random graph $G(n,p)$ (or,
equivalently, the largest independent set in $G(n,1-p)$) was
solved by \citet{BollErd76} and \citet{Mat76}, see also
\cite{Bollobas} and \cite{JLR}; again the size, $C_n$ say, is
asymptotically $2\logq n$, where $Q=1/p$.

Note that
$T_n\ge S_n \ge \floor{\tp_n/2}\ge\floor{T_n/2}$,
which shows that $T_n$, $\tp_n$ and $S_n$ are equal within a factor of
$2+o(1)$, and in particular
of the same order of magnitude. However, it does not seem possible to get
the right constant in front of $\logq n$
for one of these problems from the other.

For the largest square zero submatrix and the largest cliques in
$G(n,p)$, much
more precise estimates are known, see
\cite{SN} and \cite{Bollobas, JLR};
for example, it follows that if
\begin{equation*}
 s(n)=2\logq n-2\logq\logq n +2\logq(e/2),
\end{equation*}
 then
for any $\eps>0$,  $\floor{s(n)-\eps} \le S_n\le \floor{s(n)+\eps}$
and
$\floor{s(n)+1-\eps} \le C_n\le \floor{s(n)+1+\eps} $
\whp\ (and, in fact, almost surely);
in particular
the sizes are concentrated on one or at most two values.
It would be interesting to find similar sharper versions of the result
above, which leads to the following open problems.
\end{remark}

\begin{problem}
Find second order terms for
$T_n,\ts_n,\allowbreak \tp_n,\tps_n$, and
if possible even sharper estimates,
and see if
they differ between the four versions.
In particular, what are the orders of the
differences $\tp_n-T_n$, $T_n-\ts_n$, \dots?
\end{problem}

\begin{problem}
Are
the quantities $T_n,\ts_n,\allowbreak \tp_n,\tps_n$  concentrated on at
most two values each?
\end{problem}

\begin{problem}
Prove a version of \refT{T1} (or a stronger result) with convergence almost
surely instead of just in probability,
seeing $X_n$ as submatrices of an infinite random matrix in the natural way.
\end{problem}

\begin{problem}
Find corresponding results when $p_0$ and $p_1$ depend on $n$. The
case when $p_0$ tends to $1$ (not too fast) seems to be the most
interesting.
\end{problem}

\begin{remark}
  We consider for simplicity only square matrices $X$, but the definitions
  extend to general $m\times n$ matrices.
Since the quantities $T_n,\ts_n,\allowbreak \tp_n,\tps_n$
are monotone if we add rows
or columns, the result of \refT{T1} holds as long as $\log m/\log n\to1$;
this includes for example the case $m/n\to c\in(0,\infty)$. We have not
investigated
other cases such as $m=n^\gam$ for some $\gam>0$.
\end{remark}


\subsection{Notation}
We let $\floor x$
and $\ceil x$ denote the largest and smallest integers such that
$\floor x \le x \le \ceil x$.
We write $[m,n]$ for the interval \set{m,m+1,\dots,n} of integers between
$m$ and $n$.
Further,
$\log $ denotes the natural logaritm $\log_e$; recall that $\logq n=\log
n/\log Q$.

\section{Proof of upper bound}\label{Supper}

As said above, it suffices to show that
$\tp_n\le(2+\sqrt2+\eps)\logq n$ \whp{} for every $\eps>0$. We
will use the first moment method, \ie, show that a suitable
expectation tends to~0. However, for reasons discussed below, we
will not obtain the right constant by calculating the expected
number of (permuted) triangular submatrices of $X$. Instead we
consider the following type of submatrices.

\begin{definition}
  Let $1\le\ell\le k$. A \klcm{} is an $\ell\times\ell$ matrix
  $(a\ij)\ijto\ell$ such that
  \begin{equation}\label{corner}
a\ij=0 \qquad\text{if}\qquad i<j+k-\ell;
  \end{equation}
if further $a\ij=1$ when $i=j+k-\ell$, the matrix is a \emph{special \klcm}.
\end{definition}
Thus the $\ell\times\ell$ submatrix
in the upper right corner of a [special]
lower \qa{k} triangular matrix is a [special] \klcm, and conversely. Note
that if
$\ell\le k/2$, then a \klcm{} is 0, and if $\ell=k$ then a \klcm{} is the
same as a
triangular matrix.

Let $\nu_0(k,\ell)$ be the number of entries required to be 0 by
\eqref{corner}. Thus $\nu_0(k,\ell)=\ell^2$ when $\ell\le k/2$; for $\ell\ge
k/2$ we have
\begin{equation}
  \label{cornu}
  \begin{split}
\nu_0(k,\ell) &=\sum_{j=1}^\ell\max(j+k-\ell-1,\ell)
=\sum_{j=1}^{2\ell-k}(j+k-\ell-1) +\sum_{j=2\ell-k+1}^\ell\ell
\\&
=
\frac{(2\ell-k)(2\ell-k+1)}2 + (2\ell-k)(k-\ell-1)+(k-\ell)\ell
\\&
=\frac{4k\ell-k^2-2\ell^2+k-2\ell}2.
  \end{split}
\raisetag\baselineskip
\end{equation}
Similarly, let  $\nu_1(k,\ell)$ be the number of entries required to be 1 in
a special \klcm. Thus $\nu_1(k,\ell)=0$ when $\ell\le k/2$
and $\nu_1(k,\ell)=2\ell-k$ when $\ell\ge k/2$.
Further, let $\nu(k,\ell)\=\nu_0(k,\ell)+\nu_1(k,\ell)$ be the total number
of fixed entries in a special \klcm. If $\ell\ge k/2$, then by \eqref{cornu}
\begin{equation}  \label{cornu+}
  \nu(k,\ell)
=\frac{4k\ell-k^2-2\ell^2-k+2\ell}2.
\end{equation}

Let $1\le\ell\le m$ and let $Y_{m,\ell}$ be the number of permuted
\mlc{} submatrices in~$X$. Note that if $X$ contains a permuted
triangular $m\times m$ submatrix $A$, then a suitable submatrix of
$A$ is a permuted \mlc{} submatrix of $X$. Hence, if $\tp_n\ge m$,
then $Y_{m.\ell}\ge1$, and Markov's inequality yields
\begin{equation}
  \label{b3a}
\P(\tp_n\ge m)\le\P(Y_{m,\ell}\ge1)\le \E Y_{m,\ell} \ .
\end{equation}
The expected value $\E Y_{m,\ell}$ is easily computed.
The number of permuted $\ell\times\ell$ submatrices of $X$ is
$(n)_\ell\cdot(n)_\ell$, where $(n)_\ell\=n(n-1)\dotsm(n-\ell+1)$, and for
each such matrix, the probability that it is triangular is
$p_0^{\nu_0(m,\ell)}$, with $\nu_0(m,\ell)$ given above.
Thus,
\begin{equation}
  \label{b3}
\E \yml
= (n)_\ell^2\cdot p_0^{\nu_0(m,\ell)}
\le \exp\bigpar{2\ell\log n-\log Q\cdot\nu_0(m,\ell)}.
\end{equation}
Taking $m=\ceil{s\log n}$ and $\ell=\ceil{t\log n}$ for some fixed $s$ and
$t$ with $s/2<t\le s$, we have by \eqref{b3} and \eqref{cornu},
\begin{equation}\label{b3b}
  \E\yml \le \exp \bigpar{2t(\log n)^2-\log Q\cdot(2st-s^2/2-t^2)(\log
    n)^2+O(\log n)}.
\end{equation}
We see from \eqref{b3b} that if we choose $s$ and $t$ such that $s/2<t\le s$
and
\begin{equation}
  \label{b3c}
2t-\log Q\cdot(2st-s^2/2-t^2)<0,
\end{equation}
then $\E\yml\to0$ and thus by \eqref{b3a}
\begin{equation}
  \label{b3d}
\P(\tp_n\ge s\log n) =\P(\tp_n\ge m) \le\E\yml\to0;
\end{equation}
hence $\tp_n< s\log n$ \whp

Write for convenience $\gam\=1/\log Q$. The left hand side of
\eqref{b3c} is, for fixed~$s$, maximized when $t=s-\gam$, and then
its value is, by a short calculation,
\begin{equation*}
  2s-\gam-\frac{s^2}{2\gam}
=-\frac{s^2-4s\gam+2\gam^2}{2\gam}
=-\frac{(s-2\gam)^2-2\gam^2}{2\gam},
\end{equation*}
which is negative for $s>2\gam+\sqrt2\gam$.
Consequently, taking any $s>(2+\sqrt2)\gam$ and $t=s-\gam$, which clearly
satisfies $s/2<t<s$, \eqref{b3d} yields
$\tp_n< s\log n$ w.h.p.
It remains only to note that $\gam\log n=\log n/\log Q=\logq n$.

\begin{remark}\label{Rtribad}
  If we instead estimate the number of (permuted) triangular submatrices, we
are taking $\ell=m$ and $t=s$ in the calculations above and we
only obtain the weaker estimate $\tp_n\le(4+\eps)\logq n$ w.h.p.
The reason that the first moment method does not yield a sharp
estimate in this case is that triangular submatrices of large size
tend to occur in large clusters; thus the expected number of such
submatrices of a given size can tend to infinity although the
probability that the number is nonzero tends to~0. See also the
proof of the lower bound in \refS{Slower}, which shows that a
\klcm{} of close to maximal size \whp{} can be extended to a
triangular submatrix in many different ways.
\end{remark}

\section{Proof of lower bound}\label{Slower}

We begin by stating three lemmas; the first is elementary and the
two others contain the main probabilistic arguments. The proofs
are provided  later.

\begin{lemma}
  \label{LL0}
Suppose that $k_1\ge \ell_1\ge1$, $k_2\ge \ell_2\ge1$,  and
$2(\ell_1-\ell_2)\ge k_1-k_2\ge 0$. Then every special  \xxcm{k_1}{\ell_1}
contains a special
\xxcsm{k_2}{\ell_2}.
\end{lemma}

\begin{lemma}\label{LL1}
  Let $\eps>0$.
There exists some $k=k(n)$ and   $\ell=\ell(n)$ with
\begin{align*}
(2+\sqrt 2-\eps)\logq n&\le k\le (2+\sqrt2)\logq n,
\intertext{and}
(1+\sqrt 2-\eps)\logq n&\le \ell\le (1+\sqrt2)\logq n
\end{align*}
such that \whp{} $X$ contains a special \klcsm.
\end{lemma}

\begin{lemma}
  \label{LL2}
Let $X'$ be the submatrix $(x\ij)_{i>n/2,\,j\le n/2}$ comprising the lower
left quarter of $X$.
Let $\eps>0$ and let $k=k(n)$ and $\ell=\ell(n)$ be such that $k/2<\ell<k$ and
$k-\ell\le(1-\eps)\logq n$.
If $X'$ contains a special \klcsm, then \whp{} $X$ contains a special triangular
$k\times k$ submatrix, and thus $\ts_n\ge k$.
\end{lemma}

\begin{proof}[Proof of lower bound in \refT{T1}]
Let $0<\eps<1/3$.
Let $X'$ be the lower left quarter of $X$ as in \refL{LL2}.
By \refL{LL1}, there exists $k_1$ and $\ell_1$ with
\begin{align*}
(2+\sqrt 2-\eps)\logq \floor{n/2}&\le k_1\le (2+\sqrt2)\logq \floor{n/2},
\\
(1+\sqrt 2-\eps)\logq \floor{n/2}&\le \ell_1\le (1+\sqrt2)\logq \floor{n/2}
\end{align*}
such that there \whp{} is a special \xxcsm{k_1}{\ell_1} $M_1$ of $X'$.

Note that $k_1-\ell_1\le(1+\eps)\logq n$. Let $d=\ceil{2\eps\logq
n}$, $k=k_1-2d$, and $\ell=\ell_1-d$. By \refL{LL0}, there is a
special \klcsm{} $M_2$ of~$M_1$. It is easily verified that $k$
and $\ell$ satisfy the conditions of \refL{LL2}, and thus
\refL{LL2} shows that \whp{}
\begin{equation*}
 \ts_n\ge k\ge
(2+\sqrt 2-5\eps)\logq n +O(1).
\end{equation*}
The bound $\ts_n\ge (2+\sqrt 2-\eps)\logq n$ \whp{} follows by replacing
$\eps$ by $\eps/6$.
This completes the proof of \refT{T1} since
$\tx_n\ge\ts_n$ by \eqref{tineq}.
\end{proof}

It remains to prove the lemmas.

\begin{proof}[Proof of \refL{LL0}]
  Let $A$ be a special \klcm. The submatrix obtained by deleting the first row
  and last column is a special \xxcm{k-2}{\ell-1}.
Similarly, we obtain a special \xxcm{k-1}{\ell-1} by deleting the last row
and last column, and
a special \xxcm{k}{\ell-1} by deleting the last row
and first column.

The lemma now follows by induction on $\ell_1-\ell_2$.
\end{proof}

\begin{proof}[Proof of \refL{LL1}]
We may assume that $\eps<1/4$.
We consider a block version of \klcmm.

Let $N$ be a large integer and let $K=\ceil{(2+\sqrt 2-\eps) N}$
and $L=\ceil{(1+\sqrt 2-\eps) N}=K-N$; note that $K>L>K/2$. Let
$n_1\=\floor{n/L}$ and divide the interval \iii1{n} into the~$L$
subintervals $E_i\=\iii{(i-1)n_1+1}{in_1}$, $i=1,\dots,L$,
ignoring the possible remainder at the end.
Let $X\ij$ be the $n_1\times n_1$ submatrix $(x_{rs})_{r\in E_i,\,s\in E_j}$
of $X$.

Let
\begin{equation}
  \label{q}
q\=\ceil{N\qw \logq n}
\end{equation}
and consider the submatrices of $X$ obtained by choosing $q$ rows
from each $E_i$ and $q$ columns from each $E_j$, $i,j =1,\dots,L$.
We denote the set of all such submatrices by~$\cM$; each $M\in\cM$
is identified by its set of rows and columns, and the number of
them is thus
\begin{equation}\label{cm}
  |\cM|=\binom{n_1}q^{2L}.
\end{equation}
Each $M$ is a \qa{Lq} submatrix of $X$ which
consists of $L^2$ blocks $M\ij$, $i,j\in\set{1,\dots,L}$, where $M\ij$ is a
\qa{q} submatrix of $X\ij$.

We say that the submatrix $M\in\cM$ is \emph{good} (for a given
realization of the random matrix $X$) if $M\ij=0$ when $i<j+K-L$
and $M\ij=I$ (the \qa{q} identity matrix) when $i=j+K-L$;
otherwise $M$ is called   \emph{bad}. Thus, a good submatrix can
be seen as a special \xxcm{K}{L} of \qa{q} matrices.

Note that a good submatrix $M$ is a special \xxcm{Kq}{Lq}, and that $k=Kq$ and
$\ell=Lq$ satisfy the inequalities in the lemma if $N$ and $q$ are large
enough.
Hence it suffices to show that if $N$ is large enough, then there exists
\whp{} at least one good submatrix $M\in\cM$.

Let $I_M$ be the indicator that $M$ is good, \ie, $I_M\=1$ if $M$ is good and
$I_M\=0$ if $M$ is bad, and let $Z\=\sum_{M\in\cM}I_M$ be the number of good
submatrices $M\in\cM$. Our task is to show that $Z\ge1$ \whp, which we do by
estimating the mean and variance.

In order for $M$ to be good, the number of submatrices $M\ij$
required to be~$0$ is $\nu_0(K,L)$, and the number required to be
$I$ is $\nu_1(K,L)$. Consequently, the number of entries required
to be 0 is $\nu_0(K,L)q^2+\nu_1(K,L)(q^2-q)
=\nu(K,L)q^2-\nu_1(K,L)q$ and the number of entries required to be
1 is $\nu_1(K,L)q$. Hence, denoting the probability that $M$ is
good by $\pi$, for each $M\in\cM$,
\begin{equation}\label{pim}
\pi\=
  \P(I_M=1)
=p_0^{\nu(K,L)q^2-\nu_1(K,L)q} p_1^{\nu_1(K,L)q}
.
\end{equation}
We have by \eqref{cornu+},
recalling $K=L+N$,
\begin{equation}\label{nukl}
  \begin{split}
  \nu(K,L)
&=\frac{4(L+N)L-(L+N)^2-2L^2+O(N)}{2}
\\&
=\frac{L^2+2LN-N^2}{2}+O(N)
\\&
=\Bigpar{2+2\sqrt2-(2+\sqrt2)\eps+\frac{\eps^2}2}N^2+O(N)
<(2-\eps/2)LN   ,
  \end{split}
\end{equation}
provided $N$ is chosen large enough.
We fix such an $N$; thus $K$ and $L$ are now fixed, while \ntoo.
By \eqref{q},
\begin{equation}\label{erm}
\log n=\logq n\cdot \log Q=Nq\log Q+O(1).
\end{equation}
Furthermore,  \eqref{q} also yields,
as $\ntoo$,
$q\le\logq n\ll n_1$. Hence, by Stirling's formula,
\begin{equation*}
\log  \binom{n_1}q
= q\log n_1 + O\Bigparfrac{q^2}{n_1}-\log(q!)
= q\log n + O(q\log q).
\end{equation*}
Consequently, by \eqref{cm},  \eqref{pim}, \eqref{nukl} and \eqref{erm},
\begin{equation}\label{ez}
  \begin{split}
  \E Z&
= |\cM|\P(I_M=1)
= |\cM|\pi
\\&
=\exp\Bigpar{2L\bigpar{q\log n + O(q\log q)}-\nu(K,L)q^2\log Q+O(q)}
\\&
\ge\exp\Bigpar{2L\bigpar{q\log n}-(2-\eps/2)LN q^2\log Q+O(q\log q)}
\\&
=\exp\Bigpar{(\eps LN \log Q/2) q^2+O(q\log q)}
\to\infty
    .
  \end{split}
\end{equation}

To estimate the variance $\Var(Z)$, we first calculate the
covariance $\Cov(I_M,\imx)=\E(I_M\imx)-\E(I_M)\E(\imx)$ for two
submatrices $M,M'\in\cM$. Let $a_i$ be the number of common rows
in $E_i$ of $M$ and $M'$,
and let $b_j$ be the number of common columns in $E_j$. 
Then $M\ij$ has $a_ib_j$ entries in common with $M'\ij$, so their union
has $2q^2-a_ib_j$ elements.

For $i<j+K-L$, we have
\begin{equation}\label{mm0}
  \frac{\P(M\ij=0=M'\ij)}{\P(M\ij=0)\P(M'\ij=0)}
=\frac{p_0^{2q^2-a_ib_j}}{p_0^{2q^2}} = {p_0^{-a_ib_j}}.
\end{equation}
For $i=j+K-L$, we want $M\ij=M\ij'=I$, so we have to consider also
the required positions of the 1's in~$M\ij$ and $M'\ij$. In many
cases, the rows and columns chosen for~$M\ij$ and $M'\ij$ are such
that the conditions $M\ij=I$ and $M'\ij=I$ are contradictory, so
$\P(M\ij=M'\ij=I)=0$. Otherwise, the $a_ib_j$ common entries of
$M\ij$ and $M'\ij$ contain some number of entries, $d$ say, that
have to be 1 in both $M\ij$ and $M'\ij$, while the remaining
$a_ib_j-d$ have to be~$0$ in both, and then
\begin{equation}\label{mm1}
  \frac{\P(M\ij=M'\ij=I)}{\P(M\ij=I)\P(M'\ij=I)}
= {p_0^{-(a_ib_j-d)}}p_1^{-d}
= {p_0^{-a_ib_j}}\Bigparfrac{p_0}{p_1}^d;
\end{equation}
note that $0\le d\le \min(a_i,b_j)$.
Combining \eqref{mm0} and \eqref{mm1} by taking the product over all pairs
$(i,j)$ with $i\le j+K-L$,
and recalling that $K-L=N$,
we obtain the upper bound
\begin{equation}\label{sjw}
  \frac{\P(I_M=\imx=1)}{\P(I_M=1)\P(\imx=1)}
\le  {p_0^{-\sum_{i,j:i\le j+N}a_ib_j}}
\max\Bigpar{\Bigparfrac{p_0}{p_1}^{L\sum_i a_i},1}.
\end{equation}
Let $\pi\=\P(I_M=1)$, 
$\CC\=\max\set{(p_0/p_1)^L,1}\CCdef\CCcov$ and, for a given pair
$M,M'$, $A\=\sum_i a_i$ and $B\=\sum_j b_j$, be the numbers of
common rows and columns,
  respectively, of~$M$ and~$M'$.
Then \eqref{sjw} yields
\begin{equation}\label{jb}
\Cov(I_M,\imx)\le
  \P(I_M=\imx=1)
\le  {Q^{\sum_{i,j:i\le j+N}a_ib_j}} \CCcov^A
\pi^2.
\end{equation}

Let
\begin{equation}
  \label{tau}
\tau=\tau\bigpar{(a_i),(b_j)}\=\sum_{i,j:i\le j+N}a_ib_j
\end{equation}
and let $\tau(A,B)$ be the maximum of $\tau$ for given sums
$A=\sum_ia_i$ and $B=\sum_jb_j$, with $a_i,b_j\in\iii0q$. If
$i_1<i_2$ and we increase $a_{i_1}$ by some $\gD$ to $a_{i_1}+\gD$
and decrease $a_{i_2}$ by the same $\gD$ to $a_{i_2}-\gD$, then
$\tau\=\sum_{i,j:i\le j+N}a_ib_j$ cannot decrease. The same
happens if we decrease $b_{j_1}$ and increase $b_{j_2}$ with
$j_1<j_2$. Consequently, given $A$ and $B$, the sum~$\tau$ is
maximized when, for some indices $i_*,j_*\in\iii 1L$,
\begin{align}\label{amax}
  a_i&=q \text{ when } i<i_*, &
  a_i&=0 \text{ when } i>i_*;
\\
\label{bmax}
  b_j&=0 \text{ when } j<j_*, &
  b_j&=q \text{ when } j>j_*.
\end{align}

Returning to \eqref{jb}, we have the estimate
$\Cov(I_M,\imx)\le Q^{\tau(A,B)}\CCcov^A\pi^2$.
If $A=0$ or if $B=0$, then $M$ and $M'$ are disjoint submatrices of $X$,
and thus independent, so in this case $\Cov(I_M,\imx)=0$.
Consequently,
\begin{equation}\label{pi}
\Var(Z)=
\sum_{M,M'}  \Cov(I_M,\imx)
\le \sum_{M,M':A,B>0}
Q^{\tau(A,B)} \CCcov^A \pi^2,
\end{equation}
where $A$ and $B$ are defined as above, given $M$ and $M'$.

For a given $M\in\cM$, the number of submatrices $M'\in\cM$ with
given $a_1,\dots,a_L$, $b_1,\dots,b_L$ is
\begin{equation*}
  N\bigpar{(a_i)_i,(b_j);q}
=
\prod_{i=1}^L\binom{q}{a_i}\binom{n_1-q}{q-a_i}
\prod_{j=1}^L\binom{q}{b_j}\binom{n_1-q}{q-b_j}.
\end{equation*}
We have, for any $a\in\iii0q$,
\begin{equation}\label{lou}
  \frac{\binom{q}{a}\binom{n_1-q}{q-a}}{\binom{n_1}q}
\le
  \frac{{q}^{a}\binom{n_1-a}{q-a}}{\binom{n_1}q}
=q^a\prod_{i=0}^{a-1}\frac{q-i}{n_1-i}
\le q^a\Bigparfrac{q}{n_1}^a
=\Bigparfrac{q^2}{n_1}^a
.
\end{equation}
Thus, recalling \eqref{cm},
\begin{equation*}
\frac{  N\bigpar{(a_i)_i,(b_j);q}}{|\cM|}
\le
\Bigparfrac{q^{2}}{n_1}^{A+B}.
\end{equation*}
Moreover, given $A$ and $B$, the number of choices of $a_1,\dots,a_L$
with sum $A$ is $\le(A+1)^L\le2^{AL}$, and similarly the number of
$b_1,\dots,b_L$ is $\le 2^{BL}$.
Hence, for each $M\in\cM$, the number of $M'$ with given $A$ and $B$ is at
most, using \eqref{lou},
\begin{equation*}
2^{AL}2^{BL}\Bigparfrac{q^{2}}{n_1}^{A+B} |\cM|
=
\Bigparfrac{2^Lq^{2}}{n_1}^{A+B} |\cM|
\le
\Bigparfrac{\CC q^{2}}{n}^{A+B} |\cM|,
\CCdef\CCsw
\end{equation*}
where $\CCx=(L+1)2^L$ (for $n$ large enough).
Since $M$ can be chosen in $|\cM|$ ways, and $A,B\le Lq$, \eqref{pi} yields,
recalling $\E Z=|\cM|\pi$,
\begin{equation}\label{varz}
  \begin{split}
\Var(Z)&\le
\sum_{A,B=1}^{Lq} |\cM|
\Bigparfrac{\CCx q^{2}}{n}^{A+B} |\cM|
Q^{\tau(A,B)}\CCcov^{A+B}\pi^2
\\&=
(\E Z)^2
\sum_{A,B=1}^{Lq}
\Bigparfrac{\CC q^{2}}{n}^{A+B}
Q^{\tau(A,B)},
  \end{split}
\end{equation}
with $\CCx=\CCcov\CCsw$.
We write \eqref{varz} as
$\Var(Z)=(\E Z)^2\sum_{A,B}\gl(A,B)$, with
\begin{equation}
  \label{gl}
\gl(A,B)\=\Bigparfrac{\CCx q^{2}}{n}^{A+B} Q^{\tau(A,B)}.
\end{equation}

\begin{claim}
If $A,B\in\iii1{Lq}$, then
$\gl(A,B)\le\max\set{\gl(1,1),\gl(Lq,Lq)}$; in other words,
$\gl(A,B)$ attains its maximum for $A=B=1$ or $A=B=Lq$.
\end{claim}

To prove the claim, let $(a_i)$ and $(b_j)$ be vectors that maximize $\tau$
in \eqref{tau} for some given $A$ and $B$; we may thus assume that
\eqref{amax} and \eqref{bmax} hold.
We first note that if $A< Nq$, then by \eqref{amax} we
have $i_*\le N$ and $a_i=0$ when $i>N$; hence
\begin{equation*}
 \tau(A,B)= \tau=\sum_{i,j:i\le j+N} a_ib_j
= \sum_{i,j=1}^L a_ib_j =AB
\end{equation*}
and thus
\begin{equation*}
  \gl(A,B)=
(\CCx q^{2}/{n})^{A+B} Q^{AB}.
\end{equation*}
Keeping $A$ fixed, this is maximized by either $B=1$ or $B=Lq$.

On the other hand, if $A\ge Nq$, then \eqref{amax} yields $a_i=q$ when $i\le
N$. Hence, increasing any $b_j$ by 1 will increase $\tau$ in \eqref{tau} by
$\sum_{i:i\le j+N}a_i\ge Nq$, and thus $\tau(A,B+1)\ge \tau(A,B)+Nq$.
Consequently, by \eqref{gl} and \eqref{q},
\begin{equation*}
  \frac{\gl(A,B+1)}{\gl(A,B)} =
\Bigparfrac{\CCx q^{2}}{n} Q^{\tau(A,B+1)-\tau(A,B)}
\ge \Bigparfrac{\CCx q^{2}}{n} Q^{Nq}
\ge \CCx q^2
>1,
\end{equation*}
and thus $\gl(A,B)\le \gl(A,Lq)$ for any $B\le Lq$.

Hence, for any fixed $A\le Lq$, $\gl(A,B)$ is maximized by either $B=1$ or
$B=Lq$. By symmetry, for fixed $B$, the maximum is attained for $A=1$ or
$A=Lq$. Consequently, the maximum for all $A,B\in\iii1{Lq}$ is attained for
$A,B\in\set{1,Lq}$. Moreover, $\gl(1,Lq)=\gl(Lq,1)$ by symmetry and
$\gl(Lq,1)\le\gl(Lq,Lq)$ by the case $A\ge Nq$ above, and the claim follows.

We calculate easily the two extreme cases. For $A=B=1$, $\tau(1,1)=1$ and
\begin{equation}\label{d1}
  \gl(1,1)=
\Bigparfrac{\CCx q^{2}}{n}^2 Q
=O\Bigparfrac{\log^{4}n}{n^2}.
\end{equation}
For $A=B=Lq$, all $a_i=b_j=q$, and thus $\tau(Lq,Lq)=\nu(K,L)q^2$.
Hence, recalling $q=O(\log n)$, \eqref{erm} and \eqref{nukl},
\begin{equation}\label{d2}
  \begin{split}
  \gl(Lq,Lq)&=
\Bigparfrac{\CCx q^{2}}{n}^{2Lq} Q^{\nu(K,L)q^2}
\\&
=\exp\Bigpar{-2Lq\log n+ O(q\log q)+\nu(K,L)q^2\log Q}
\\&
=\exp\Bigpar{(-2LNq^2+\nu(K,L)q^2)\log Q+O(q\log q)}
\\&
\le
\exp\Bigpar{-(\eps LN \log Q/2)q^2+ O(q\log q)}.
  \end{split}
\end{equation}
For large $n$, this is less than $\exp(-2Nq)<n\qww$.
Consequently, the claim and  \eqref{d1}--\eqref{d2} shows that for all
$A,B\le Lq$,
\begin{equation}\label{glfin}
  \gl(A,B)= O\Bigparfrac{\log^{4}n}{n^2}.
\end{equation}
Finally,  by \eqref{varz} and \eqref{glfin},
\begin{equation}\label{zgood}
  \frac{\Var(Z)}{(\E Z)^2}
\le \sum_{A,B=1}^{Lq}\gl(A,B)
=O\Bigparfrac{q^2\log^{4}n}{n^2}
=O\Bigparfrac{\log^{6}n}{n^2}=o(1),
\end{equation}
as \ntoo. This is what we need: by Chebyshev's inequality
\begin{equation*}
\P(Z=0)\le \frac{\Var(Z)}{(\E Z)^2};
\end{equation*}
hence \eqref{zgood} yields
$\P(Z=0)\to0$, and thus $Z\ge1$ \whp, which completes the proof.
\end{proof}

\begin{proof}[Proof of \refL{LL2}]
Condition on $X'$ and
fix a special \klcsm{} $M'=(x_{i'_r,j'_s})_{r,s=1}^\ell$ of $X'$; thus
$n/2<i'_1<\dots<i'_\ell\le n$ and
$1\le j'_1<\dots<j'_\ell\le n/2$.
We try to complete $M'$ to a $k\times k$ special triangular matrix by
adding $k-\ell$ rows $i_1<\dots<i_{k-\ell}\le n/2$ and
$k-\ell$ columns $n/2<j_1<\dots<j_{k-\ell}\le n$; we do this by trying the
rows one by one until we find first a suitable $i_1$
(i.e., one with $x_{i_1j'_1}=1$),
then a suitable $i_2$
(one with $x_{i_2j'_1}=0$ and $x_{i_2j'_2}=1$), and so on until
$i_{k-\ell}$, and similarly for $j_1,\dots,j_{k-\ell}$.

Let $r\le k-\ell$.
Each time we try a row in order to find $i_r$, we want one specific entry
in it to be 1 and $r-1$ others to be 0; the probability of this is
$\pi_r\=p_0^{r-1}p_1$, independently of $X'$ and what has happened earlier.
If $T_r$ is the number of rows that we have to try until we find $i_r$,
then $T_r$ thus has a geometric distribution
\begin{equation*}
  \P(T_r=t)=(1-\pi_r)^{t-1}\pi_r,
\qquad t=1,2,\dots.
\end{equation*}
This distribution has mean $\E T_r= 1/\pi_r$ and variance
$\Var T_r= (1-\pi_r)/\pi_r^2$;
hence the sum $S:=T_1+\dots+T_{k-\ell}$ has mean
\begin{align*}
  \E S&
=\sum_{r=1}^{k-\ell}\E T_r
=\sum_{r=1}^{k-\ell}\pi_r\qw
=\sum_{r=1}^{k-\ell} p_1\qw Q^{r-1}
=O\bigpar{Q^{k-\ell}}
= O\bigpar{n^{1-\eps}} = o(n)
\intertext{and variance}
 \Var S&
=\sum_{r=1}^{k-\ell}\Var T_r
\le \sum_{r=1}^{k-\ell}\pi_r\qww
=O\bigpar{Q^{2(k-\ell)}}
= O\bigpar{n^{2(1-\eps)}} = o(n^2).
\end{align*}
The search for $i_1,\dots,i_{k-\ell}$ succeeds if $S\le n/2$. Consequently
the probability of failure is, using Chebyshev's inequality, for $n$ so
large that $\E S<n/4$,
\begin{equation*}
  \begin{split}
\P(S>n/2)
\le \frac{\Var S}{(n/2-\E S)^2}
\le \frac{\Var S}{(n/4)^2}
=o(1).
  \end{split}
\end{equation*}
Hence, \whp{} we succeed and find suitable rows $i_1,\dots,i_{k-\ell}$;
similarly \whp{} we find also suitable columns
$j_1,\dots,j_{k-\ell}$, and we can extend $M'$ to a special triangular
$k\times k$ matrix.
\end{proof}

Note that \whp{} $S$ is much less than $n/2$, so we have a wide margin in
this proof and there are \whp{} many different choices of rows and columns
that work, and thus many different ways to extend $M'$ to a special
triangular matrix, \cf{} \refR{Rtribad}.

\newcommand\AAP{\emph{Adv. Appl. Probab.} }
\newcommand\JAP{\emph{J. Appl. Probab.} }
\newcommand\JAMS{\emph{J. \AMS} }
\newcommand\MAMS{\emph{Memoirs \AMS} }
\newcommand\PAMS{\emph{Proc. \AMS} }
\newcommand\TAMS{\emph{Trans. \AMS} }
\newcommand\AnnMS{\emph{Ann. Math. Statist.} }
\newcommand\AnnPr{\emph{Ann. Probab.} }
\newcommand\CPC{\emph{Combin. Probab. Comput.} }
\newcommand\JMAA{\emph{J. Math. Anal. Appl.} }
\newcommand\RSA{\emph{Random Struct. Alg.} }
\newcommand\ZW{\emph{Z. Wahrsch. Verw. Gebiete} }
\newcommand\DMTCS{\jour{Discr. Math. Theor. Comput. Sci.} }

\newcommand\AMS{Amer. Math. Soc.}
\newcommand\Springer{Springer-Verlag}
\newcommand\Wiley{Wiley}

\newcommand\vol{\textbf}
\newcommand\jour{\emph}
\newcommand\book{\emph}
\newcommand\inbook{\emph}
\def\no#1#2,{\unskip#2, no. #1,} 
\newcommand\toappear{\unskip, to appear}

\newcommand\urlsvante{\url{http://www.math.uu.se/~svante/papers/}}
\newcommand\arxiv[1]{\url{arXiv:#1.}}

\def\nobibitem#1\par{}

\end{document}